\documentclass[11pt]{amsart}
\usepackage[margin=1in]{geometry}
\usepackage{graphicx}
\usepackage{caption,subcaption}
\usepackage{xcolor}
\usepackage{comment}
\usepackage{float}
\usepackage{amssymb,amsmath}
\usepackage{epstopdf}
\usepackage{enumerate}
\usepackage[title]{appendix}

\usepackage[colorlinks=true, pdfstartview=FitV, urlcolor=blue]{hyperref}

\newcommand{\unif}{\mathrm{unif}}

\newtheorem{theorem}{Theorem}[section]

\newtheorem{lemma}[theorem]{Lemma}

\newtheorem{remark}[theorem]{Remark}

\title{Root Dynamics of Differentiated Polynomials with Rotationally Invariant Structure}
\author{Joseph Najnudel }
\author{Truong Vu}
\thanks{TV is supported in part through Institute for Pure and Applied Mathematics (IPAM) by the National Science Foundation (Grant No. DMS-1925919/2422832)}
\date{June 2025}

\begin{document}
\begin{abstract}
    The dynamics of polynomial roots under repeated differentiation has recently been conjectured to converge to a limiting measure governed by specific nonlinear PDEs, the conjectures being shown in some particular settings. For rotationally invariant initial distributions, a deterministic structured sampling model placing roots on concentric circles was recently introduced by Galligo, Najnudel, and Vu (2025) \cite{galligo2025dynamics}. In this paper, the authors proved convergence under the technical growth condition $m_n / (n \log n) \to \infty$, where $n$ is the number of circles and $m_n$ is the number of points per circle. In this paper, we significantly improve this result by relaxing the growth condition to $m_n / \log n \to \infty$, thus allowing for regimes where the number of points per circle grows proportionally to the number of circles. The key innovation is a refined upper bound on the root magnitudes after differentiation. This sharper estimate prevents the rapid accumulation of errors over multiple differentiations, fully validating a recent conjecture regarding the robustness of the sampling scheme.
\end{abstract}
\maketitle
\section{introduction}
The study of zero dynamics under repeated differentiation has emerged as a rich area of interaction between analysis, probability, and combinatorics. For a monic polynomial $P_N(z) = \prod_{k=1}^N (z - z_k)$ with roots $z_1,\ldots,z_N \in \mathbb{C}$, a fundamental question asks how the empirical distribution of the roots of the $k$-th derivative $P_N^{(k)}$ evolves as $N \to \infty$ and $k$ grows with $N$. When the roots are sampled independently from a fixed distribution $\mu_0$ supported on the real line, it is known that the roots of the $\lfloor tN\rfloor$-th derivative converge to a limiting measure $\mu_t$ depending only on $\mu_0$ and $t \in (0,1)$, characterized by the Steinerberger PDEs~\cite{alazard-lazar-nguyen,steinerberger2019nonlocal,kiselev-tan}.  We refer to \cite{galligo2024anti, galligo2025dynamics, michelen2022zeros, michelen2023almost, feng2019zeros, arizmendi2024finitefreecumulantsmultiplicative, kabluchko2021repeated,steinerberger2023free} and references therein for related developments. 

The two-dimensional analogue, where $\mu_0$ is rotationally invariant on $\mathbb{C}$, has attracted considerable attention. O'Rourke and Steinerberger~\cite{o2021nonlocal} conjectured a nonlocal PDE governing the evolution of the radial density, equivalent to the quantile relation
\begin{align*}
    \frac{\Psi_t^{(-1)}(x)}{x} = \frac{\Psi_0^{(-1)}(x+t)}{x+t}, \qquad 0 < x < 1-t,
\end{align*}
where $\Psi_t$ denotes the cumulative distribution function of the radial part of $(1-t)\mu_t$. This conjecture has been supported by formal arguments~\cite{campbell2024fractional} and proven for specific classes of random polynomials: Hoskins and Kabluchko~\cite{hoskins2020dynamics} established it for polynomials with independent coefficients, while Feng and Yao~\cite{feng2019zeros} obtained the first rigorous result in the radial regime for random polynomials with independent coefficients.

In a recent work, Galligo, Najnudel, and Vu~\cite{galligo2025dynamics} considered a deterministic structured sampling of a rotationally invariant measure $\mu_0$. For a sequence $(m_n)_{n \geq 1}$ of positive integers, they study polynomials of the form
\begin{align*}
    P_{n,m_n}(z) = \prod_{j=1}^n (z^{m_n} - (r_j^{(n)})^{m_n}) = \prod_{j=1}^n \prod_{k=0}^{m_n-1} (z - r_j^{(n)} e^{2\pi i k/m_n}),
\end{align*}
where $(r_j^{(n)})_{1\leq j\leq n}$ samples the radial part $\nu_0$ of $\mu_0$. This construction places roots on $n$ concentric circles, each containing $m_n$ equidistributed points, which approximates rotational invariance. Their main result (Theorem 4.1 in \cite{galligo2025dynamics}) establishes that under the growth condition
\begin{align*}
    \frac{m_n}{n\log n} \xrightarrow[n\to\infty]{} \infty,
\end{align*}
the empirical measure of the roots of the $\lfloor n m_n t\rfloor$-th derivative converges to $\mu_t = \nu_t \otimes \text{unif}$, where $\nu_t$ is given by the quantile relation above.

The condition $m_n/(n\log n) \to \infty$ arises from technical estimates controlling the propagation of errors through successive differentiations. Specifically, when consecutive radii are tightly clustered (i.e., the ratio $\max_j r^{(n)}_j/r^{(n)}_{j+1}$ is close to $1$), the bounds in Lemmas 4.3 and 4.4 of \cite{galligo2025dynamics} deteriorate. To circumvent this, the authors employ a regularization technique using minorant and majorant sequences scaled by successive powers of $\gamma_n = e^{3m_n^{-1}\log n}$, which forces $\alpha \leq \gamma_n^{-1}$. The requirement $m_n/(n\log n) \to \infty$ then ensures $\gamma_n^n \to 1$, preserving the empirical measure at first order. The authors conjecture that this condition can be relaxed, and that the conclusion should hold even when $m_n$ is of the same order as $n$, or under weaker assumptions.

In this paper, we improve the main result of \cite{galligo2025dynamics} by significantly weakening the growth condition. We prove that convergence to $\mu_t$ holds under the much milder assumption
\[
\frac{m_n}{\log n} \xrightarrow[n\to\infty]{} \infty,
\]
allowing for sequences where $m_n$ grows only slightly faster than $\log n$, and in particular covering the regime $m_n \sim c n$ for any $c>0$. The key innovation lies in sharpening the upper bound estimates for the roots after differentiation, which reduces the accumulation of error factors over multiple differentiation steps. By carefully propagating these improved estimates through $\ell m_n$ differentiations (Lemma~\ref{lem1.9}) for $1 \leq \ell \leq n-1$, we obtain upper bounds that remain controlled as soon as $m_n/\log n \to \infty$.

Indeed, the improved bounds allow us to relax the regularization parameter $\gamma_n$ to
\begin{align*}
    \gamma_n = \exp\left(\frac{\phi(n)\log n}{n m_n}\right),
\end{align*}
with $\phi(n) \to \infty$ arbitrarily slowly. This choice satisfies $\gamma_n^n = n^{\phi(n)/m_n} \to 1$ under the condition $m_n/\log n \to \infty$, whereas the original argument required $\gamma_n^n \to 1$ with the stronger factor $e^{3n\log n/m_n}$. The rest of the proof is structured as follows: we construct minorant and majorant sequences for 
the radii of roots of iterated derivatives 
which have controlled spacing, apply the monotonicity properties established in Section 2 of \cite{galligo2025dynamics} to preserve ordering under differentiation, and leverage the improved bounds to show that the empirical measure of the roots converges to the limit $\nu_t$. The quantile relation and the PDEs under additional regularity assumptions are derived. 

Our result confirms that the limiting measure $\mu_t$ is robust to the sampling scheme. It also validates the conjecture of \cite{galligo2025dynamics} that their theorem holds under weaker assumptions, including the case where $m_n$ has the same order of magnitude as $n$. The techniques developed here may be applicable to other problems involving the dynamics of polynomial roots under differentiation.

   \begin{figure}[H]
\begin{subfigure}{.5\linewidth}
 \centering
  \includegraphics[width=\linewidth]{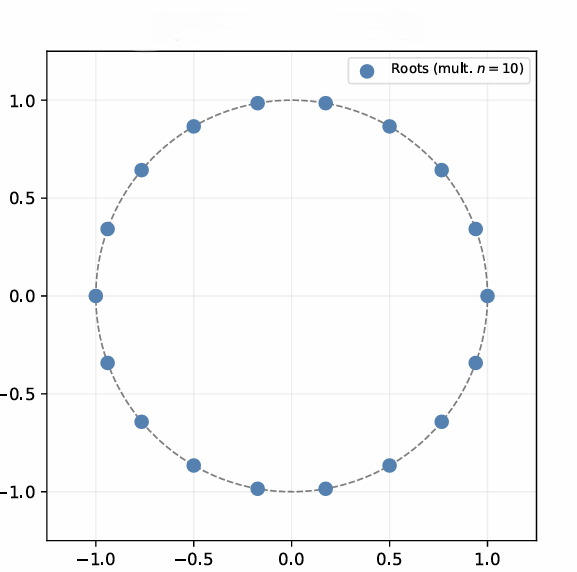}
  \caption{$t=0$ and $nm_n=180$}
  \label{}
\end{subfigure}
\begin{subfigure}{.5\linewidth}
 \centering
\includegraphics[width=\linewidth]{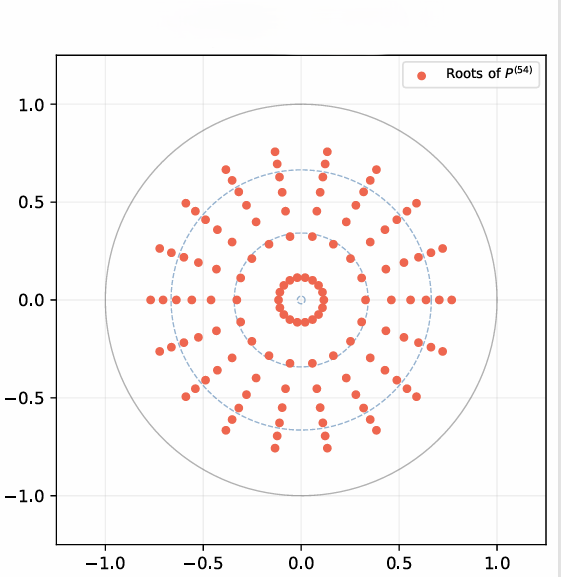}
    \caption{$t=0.3$ and  $\lfloor nm_nt\rfloor=54$ ($7$ radii, $126$ roots)}
    \label{fig:placeholder}
    \end{subfigure}
    \caption{Initial root configuration ($t=0$, degree $nm_n=180$, roots on circles with multiplicity $n=10$) alongside roots of $P^{(54)}$ at $t=0.3$ ($7$ radii, $126$ roots)} 

\end{figure}

   \begin{figure}[H]
\begin{subfigure}{.5\linewidth}
 \centering
  \includegraphics[width=\linewidth]{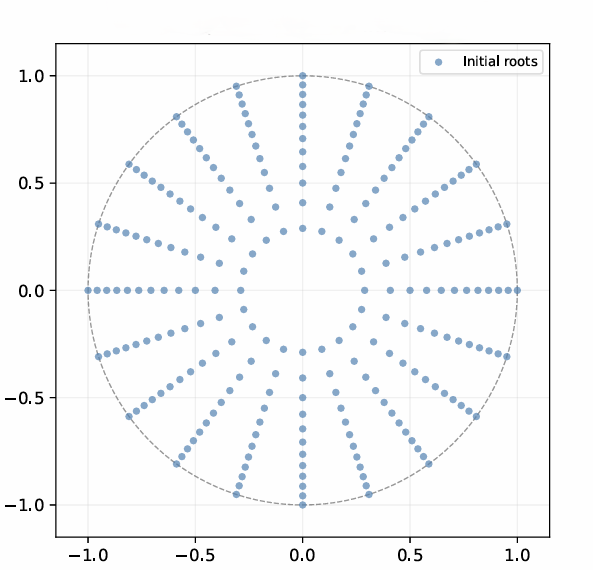}
  \caption{t=0 with degree $nm_n=240$}
  \label{}
\end{subfigure}
\begin{subfigure}{.5\linewidth}
 \centering
  \includegraphics[width=\linewidth]{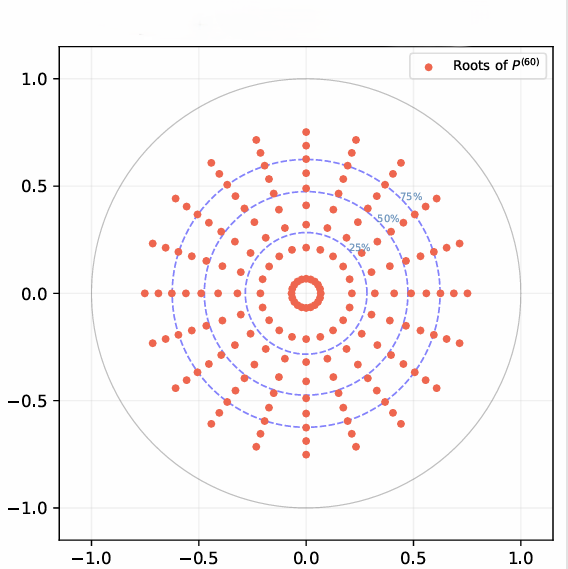}
    \caption{t=$0.25$ with $\lfloor nm_nt\rfloor=60$ and $9$ circles}
    \label{fig:placeholder}
    \end{subfigure}
    \caption{Initial root configuration ($t=0$, degree $nm_n=240$) alongside roots of $P^{(60)}$ at $t=0.25$ ($9$ circles)}
\end{figure}

  \begin{figure}[H]
\begin{subfigure}{.5\linewidth}
 \centering
  \includegraphics[width=\linewidth]{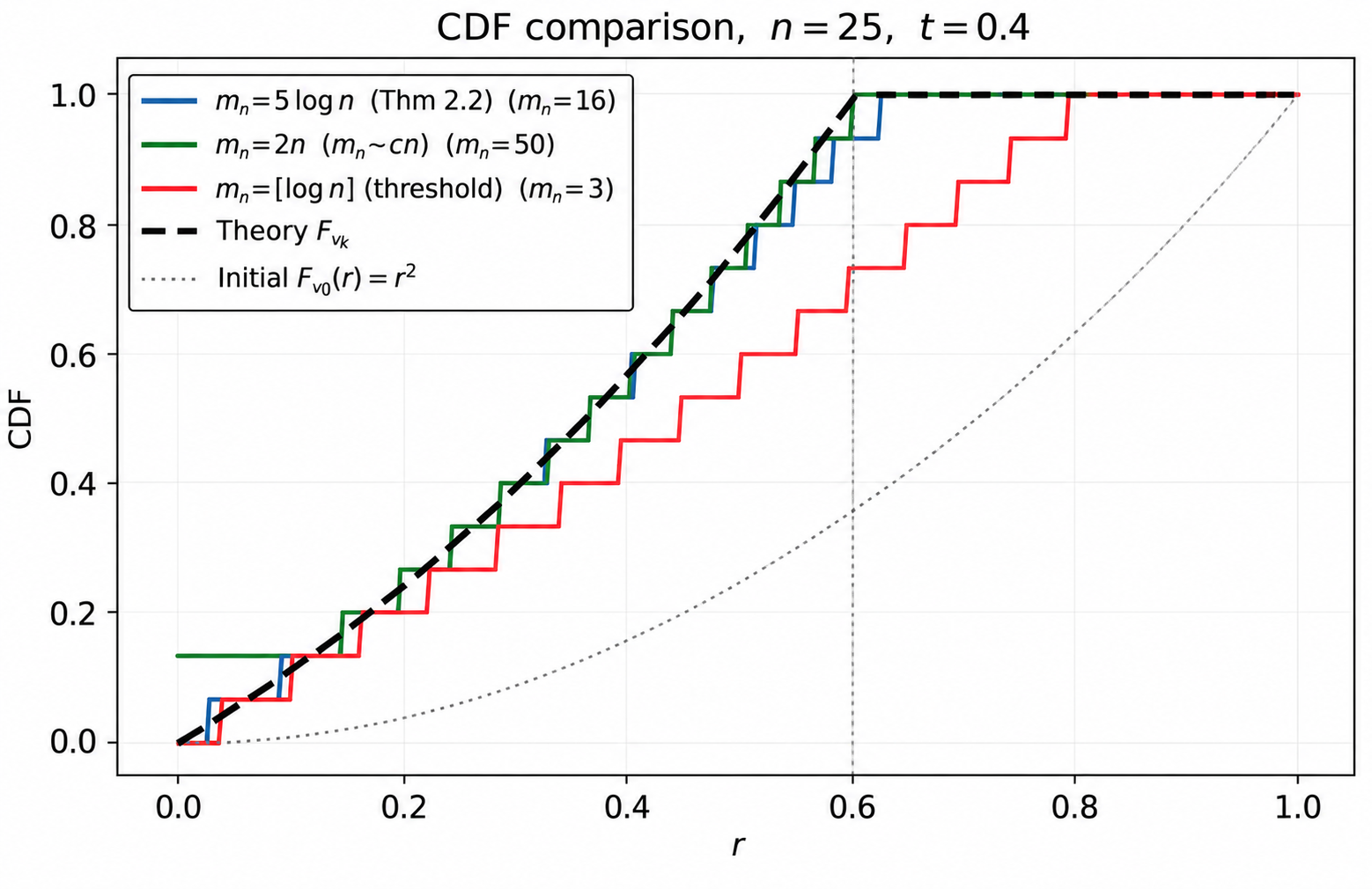}
  \caption{CDF comparison with $n=25$ and $t=0.4$}
  \label{}
\end{subfigure}
\begin{subfigure}{.5\linewidth}
 \centering
  \includegraphics[width=\linewidth]{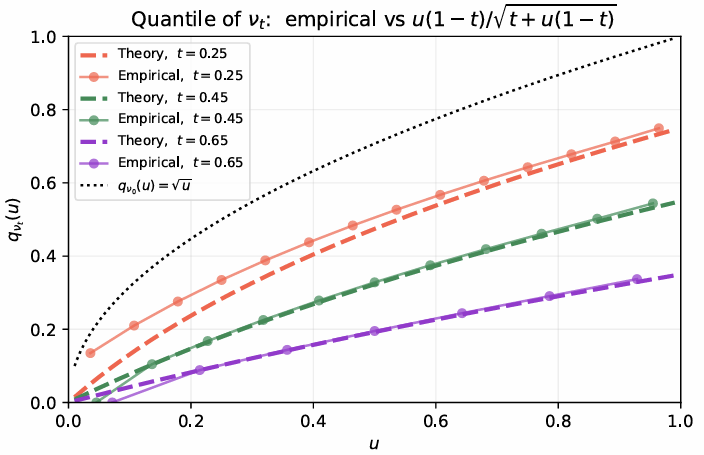}
    \caption{Quantile of $\nu_t$ with $n=20$ and $m_n=40$ }
    \label{fig:placeholder}
    \end{subfigure}
    \caption{}
\end{figure}

The paper is organized as follows. 
Section \ref{sec:sing-Step} contains the core technical innovation: we derive refined upper bounds for a single differentiation step, exploiting the specific structure of rational equations involved in order to obtain sharper estimates on the root modulii. In Section \ref{sec:iterated_diff}, we propagate these single-step bounds across multiple iterated differentiations, showing that the error accumulation remains controlled under our weakened growth conditions. Section \ref{sec:mainthm} presents the formal statement and proof of our main convergence result (Theorem \ref{thm:41improJ}), demonstrating the weak convergence of the empirical measure to $\mu_t$ and detailing the resulting quantile relations. Finally, Section \ref{sec:conclude} offers concluding remarks, summarizing the resolution of the robustness conjecture and discussing open avenues for future work.

\section{Refined Single-Step Estimates}\label{sec:sing-Step}
In this section, we establish tight bounds on the roots of the derivative of our structured polynomials. When differentiating the polynomial, the new roots are governed by a rational equation. To prevent the accumulation of errors over multiple differentiation steps, we must establish sharp bounds on $z_j$ relative to the initial radii $r_j$. The lower bound follows similar arguments to those in \cite{galligo2025dynamics}, but the upper bound requires a novel approach to be refined. In particular, we can improve the error bounds in Lemma 4.3 in \cite{galligo2025dynamics} by exploiting more carefully the particular structure of the rational sum providing the roots of polynomials of the form $S(z) = q Q(z) + m z Q'(z)$, 
for a suitable polynomial $Q$ and suitable integers $m$ and $q$. We can then propagate these bounds through iterated differentiations, modify Lemma 4.6 in \cite{galligo2025dynamics}, and prove a convergence result similar to Theorem 4.1 in \cite{galligo2025dynamics}.
\medskip
 
\begin{lemma}\label{impr_upper_bnd}
Let $Q$ be a polynomial of degree $n \geq 1$ with positive real roots, and let 
$(r_j)_{1 \leq j \leq n}$ be a strictly increasing sequence of positive real numbers. We assume that for $\alpha \in (0,1)$, 
$$ \max_{1 \leq j \leq n-1} (r_j/r_{j+1}) \leq \alpha$$
and we define
$$y := \frac{2 + \log_- (\log (\alpha^{-1}))}{\log (\alpha^{-1})},$$
where $\log_-$ denotes the negative part of $\log$, i.e. the maximum of $0$ and $- \log $. 
Suppose the $ j $-th smallest root of $ Q $ is at most $ r_j $ for $ 1 \leq j \leq n $. Then, for integers $m \geq 1$, $ 1 \leq q \leq m-1 $ and $ 1 \leq j \leq n $, all roots of the polynomial $S$ given by 
$$S(z) = q Q(z) + m z Q'(z)$$
are real and positive, and the $ j $-th smallest root of $S$ 
is at most  
\begin{align*}
\left(1 - \frac{\alpha}{  j+2 + 2y }\right) r_j.
\end{align*}  
For $ 2 \leq j \leq n $, the $ (j-1) $-th smallest root of $ Q' $, which has all real and positive roots, is at least  
\begin{align*}
\left(1 - \frac{\alpha}{  j+2 +2y }\right) r_j.
\end{align*}
\end{lemma}
\begin{proof}
The fact that all roots of $S$ and $Q'$ are real and positive is proven in \cite{galligo2025dynamics}, Lemma 3.1. 
We give the proof for the upper bound on roots of $S$: the proof for $Q'$ is analogous. 

\medskip\noindent
By the monotonicity property proven in Section 2 of ~\cite{galligo2025dynamics}, we can assume that 
the $j$-th smallest root of $Q$ is exactly $r_j$.
In this case, by Lemma~3.1 in \cite{galligo2025dynamics} the roots $z_j$ of $S$ interlace with the roots $r_j$ of $Q$: 
\begin{align*}
0< z_1<r_1<z_2<r_2<\dots<z_n<r_n .
\end{align*}
The rational function  
\begin{align*}
R(z)=\frac{S(z)}{Q(z)}
      =q+m\,z\frac{Q'(z)}{Q(z)}
      =q+m\sum_{\ell=1}^{n}\frac{z}{z-r_{\ell}}
\end{align*}
satisfies $R(z_j)=0$ for every $j$.
Since $z_j \neq 0$, we deduce 
$$\frac{q}{z_j} + \sum_{\ell =1}^n 
\frac{m}{z_j - r_{\ell}} = 0$$
i.e. 
$$\frac{q}{z_j} + \sum_{1 \leq \ell \leq n, \ell \neq j-1, j}  
\frac{m}{z_j - r_{\ell}} 
+ \frac{m}{z_j - r_{j-1}} \mathbf{1}_{j \geq 2} + \frac{m}{z_j - r_j} =  0,$$
which implies 
$$\frac{q}{z_j} + \sum_{1 \leq \ell \leq   j -2 }  
\frac{m}{z_j - r_{\ell}} 
+   \frac{m}{z_j - r_{j-1}}  \mathbf{1}_{j \geq 2} + \frac{m}{z_j - r_j}  \geq 0.$$
For $\ell \leq j - 2$, $r_{\ell} 
\leq r_{j-1} \alpha^{j-1-\ell}
\leq z_j\alpha^{j-1-\ell}$, 
$z_j - r_{\ell}
\geq z_j(1 - \alpha^{j-1-\ell})$, 
$$\frac{m}{z_j - r_{\ell}}
\leq \frac{m}{ z_j (1 - \alpha^{j-1-\ell})}.$$
 Hence, 
$$\frac{1}{z_j} 
\left( q + \sum_{1 \leq \ell \leq j-2}
\frac{m}{1 - \alpha^{j-1-\ell}} \right)
+ \frac{m}{z_j - r_{j-1}}  \mathbf{1}_{j \geq 2} + \frac{m}{z_j - r_j}  \geq 0.$$
We have for $p \geq 1$, 
 \begin{align*}
\frac{1}{1-\alpha^{p}}
   =\frac{\alpha^{-p}}{\alpha^{-p}-1}
   =1+\frac{1}{\alpha^{-p}-1}.
\end{align*}
Hence, 
\begin{align*}
\sum_{1 \leq p \leq j-2}\frac{1}{1-\alpha^{p}}
  & =\max(0,j-2)+\sum_{1 \leq p \leq j-2}\frac{1}{\alpha^{-p}-1}\\
   &\le j-1+\sum_{p=1}^{\infty}\frac{1}{\alpha^{-p}-1}
\end{align*}
since $\frac{1}{\alpha^{-p}-1}>0$.

Now, the last infinite sum is bounded by 
$$\frac{1}{\alpha^{-1} - 1}
+ \int_{1}^{\infty} \frac{1}{ e^{ x \log (1/\alpha)} - 1} dx
= \frac{1}{\alpha^{-1} - 1}
+ \int_{ \log (1/\alpha)}^{\infty} \frac{dy}{\log (1/\alpha) (e^{y} - 1)} 
$$ $$\leq \frac{1}{\alpha^{-1} - 1} 
+ \frac{1} {\log (1/\alpha)} \left(\int_{\min(\log (1/\alpha), 1)}^1 \frac{dy}{y} + \int_1^{\infty} \frac{dy}{e^y/2} \right) 
$$ $$ \leq \frac{   2 + 
\log_- ( \log (\alpha^{-1}))  }{\log (\alpha^{-1})} = y$$
Hence, 
$$\frac{1}{z_j} 
\left( q + m (j-1+y)    \right)
+ \frac{m}{z_j - r_{j-1}} \mathbf{1}_{j \geq 2}+ \frac{m}{z_j - r_j}  \geq 0.$$
We know that $0< q < m$, and then, 
dividing by $m$, 
$$\frac{1}{z_j} 
 (j  + y) 
+ \frac{1}{z_j - r_{j-1}} \mathbf{1}_{j \geq 2}+ \frac{1}{z_j - r_j}  > 0.$$
If $j \geq 2$ and $z_j > r_j \alpha$, then 
$z_j - r_{j-1} > z_j - r_j \alpha \geq  0$
and 
$$\frac{1}{ z_j - r_{j-1}} \mathbf{1}_{j \geq 2} \leq 
\frac{1}{z_j - r_j \alpha}.$$
 This inequality is obviously also true for $j = 1$. 
We deduce 
$$\frac{j+y}{r_j \alpha} 
+ \frac{1}{z_j - r_j \alpha}+ \frac{1}{z_j - r_j}  > 0.$$
We get that $z_j/r_j$ is 
at most the root $x \in (\alpha,1)$
of the equation
$$
a + \frac{1}{x -   \alpha} + \frac{1}{x - 1}  = 0.$$
for 
$$a = \frac{j+ y}{  \alpha}.$$
Since the equation has no root larger than $1$ because the left-hand side is positive for $x > 1$, 
we have to take the largest root of the second degree equation
$$ a(x-\alpha)(x-1) + 2x-1- \alpha = 0.$$
This root is given by
$$ x = \frac{ a(1 + \alpha) -2 + \sqrt{4 + a^2 (1-\alpha)^2}}{ 2a}.$$
Hence, 
\begin{align*}
1 - x & =  \frac{ a(1 - \alpha) +2 - \sqrt{4 + a^2 (1-\alpha)^2}}{ 2a}
 \\& = \frac{ [a(1 - \alpha) +2]^2- (4 + a^2 (1-\alpha)^2) }{ 2a[a(1 - \alpha) +2 + \sqrt{4 + a^2 (1-\alpha)^2}]}
  \\ & = \frac{ 2 (1-\alpha)}{  a(1 - \alpha) +2 + \sqrt{4 + a^2 (1-\alpha)^2}}.
\end{align*}
Using the subadditivity of the square root, we deduce 
$$ 1 - x \geq \frac{2(1-\alpha)}{ 2 a (1-\alpha) + 4} = \frac{1}{a + 2/(1-\alpha)} = 
\left (  \frac{j}{  \alpha}
+  \frac{2 + \log_- (\log (\alpha^{-1}))}{\alpha \log (\alpha^{-1})} + \frac{2}{1-\alpha}
\right)^{-1}. 
$$
Moreover, 
$$\frac{1}{ \log (\alpha^{-1})} 
= \frac{1}{\sum_{r \geq 1} (1 - \alpha)^r/r}
\geq \frac{1}{\sum_{r \geq 1} (1 - \alpha)^r }
= \frac{1}{ (1- \alpha)/\alpha}
= \frac{\alpha}{1-\alpha} = \frac{1}{1- \alpha} - 1,
$$
$$\frac{2}{1 - \alpha} \leq \frac{2}{\log (\alpha^{-1})} + 2,$$
which gives, since $1 \leq 1/\alpha$, 
$$ 1 - x  \geq \left (  \frac{j+2}{  \alpha}
+  \frac{4 + \log_- (\log (\alpha^{-1}))}{\alpha \log (\alpha^{-1})}  
\right)^{-1} 
\geq \frac{\alpha}{j + 2 + 2y}.
$$
We then get 
$$z_j/r_j \leq  1- \frac{\alpha}{j + 2 + 2y} $$
under the assumption above that this ratio is larger than $\alpha$. This assumption can be discarded since the upper bound just above is itself larger than $\alpha$. 
Indeed, since $j \geq 1$,
$$j + 2 + 2y 
\geq 3 + \frac{4}{\log (\alpha^{-1})} 
\geq 3 + \frac{4}{1 - \alpha} - 4
= \frac{4}{1 -\alpha} - 1 \geq \frac{3}{1- \alpha},$$
and 
$$1 - \frac{\alpha}{j + 2 + 2y}
\geq 1 - \frac{\alpha(1 - \alpha)}{3}
> \alpha. 
$$

\end{proof}

The following result is almost a rewriting of \cite{galligo2025dynamics}, Lemma 4.4, so we do not prove it again. The only difference is that in the present article, $\alpha$ is at least
the maximum of $r_{j}/r_{j+1}$, but not necessarily exactly this maximum. 
This change is not an issue, since the lower bound on the roots is nonincreasing with respect to $y$, and the nonincreasing in $\alpha$. 

\begin{lemma}\label{lower_bnd}
Let $Q$ be a polynomial of degree $n \geq 1$ with positive real roots, and let 
$(r_j)_{1 \leq j \leq n}$ be a strictly increasing sequence of positive real numbers, such that 
$$ \max_{1 \leq j \leq n-1} (r_j/r_{j+1}) \leq \alpha.$$
Suppose the $ j $-th smallest root of $ Q $ is at least $ r_j $ for $ 1 \leq j \leq n $. Then, for integers $m \geq 1$, $ 1 \leq q \leq m-1 $ and $ 1 \leq j \leq n $, the $ j $-th smallest root of the polynomial $S$ given by 
$$S(z) = q Q(z) + m z Q'(z)$$
is at least
\begin{align*}
\left(1 - \frac{1}{ \max(1, j - 1 - y)}\right) r_j.
\end{align*}  
For $ 2 \leq j \leq n $, the $ (j-1) $-th smallest root of $ Q' $ is at least  
\begin{align*}
\left(1 - \frac{1}{ \max(1, j - 1 - y)}\right) r_j.
\end{align*}
\end{lemma}

\section{Iterated Differentiations and Error Propagation}\label{sec:iterated_diff}

In the previous section, we established refined bounds for a single differentiation step. To analyze the root dynamics' macroscopic evolution, we must track these estimates across multiple differentiations. 
The following result provides bounds on the radii of the circles after a number of differentiations equal to the number of roots on each circle.

\begin{lemma}[Improved bounds after $m$ differentiations]\label{lemma:improved-m-steps}
For integers $m, n \geq 1$ and nondecreasing sequences of positive radii
$(r_j^{(n)})_{1 \leq j \leq n}$, 
we consider the polynomial $P_{n,m}(z)=\prod_{j=1}^{n}\bigl(z^{m}-(r_j^{(n)})^{m}\bigr)$. 
Then, the $m$-th derivative of $P_{n,m}$ can be written as 
$$P_{n,m}^{(m)} (z) = \frac{(nm)!}{((n-1)m)!} \prod_{j = 1}^{n-1} (z^m - (R_j^{(n)})^m) $$
for a nondecreasing sequence of  $(R_j^{(n)})_{1\le j\le n-1}$ of positive numbers. Let $(r_j)_{1 \leq j \leq n}$ be an increasing sequence of positive real numbers, and let $\alpha$ be larger than or equal to the maximum of $r_j/r_{j+1}$
for $1 \leq j \leq n-1$. If $r_j^{(n)} \leq r_j$ for $1 \leq j \leq n$, 
then for $1 \leq j \leq n-1$, 
$$ R_j^{(n)} \leq \left(1 - \frac{\alpha^m}{j + 3 + 2 y_m} \right)
r_{j+1}$$
where
\[
y_m=\frac{2+\log_{-}{(m\log(\alpha^{-1}))}}{m\log(\alpha^{-1})}.
\]
On the other hand, if $r_j^{(n)} \geq r_j$ for $1 \leq j \leq n$, then for $1 \leq j \leq n-1$, 
$$ R_j^{(n)} \geq \left(1 - \frac{1}{\max(1, j-1 -  y_m)} \right)
r_{j+1}.$$
 
\end{lemma}
\begin{proof}
The lower bound has been already proven in \cite{galligo2025dynamics}, Lemma 4.5, with the only difference that $\alpha$ is not necessarily 
exactly the maximum of $r_{j}/r_{j+1}$, which is not an issue since the lower bound is nonincreasing in $\alpha$.

Let us prove the upper bound.

We have 
$$P_{n,m}(z) = Q_0(z^m)$$
where 
$$Q_0(z) = \prod_{j=1}^n (z - (r_j^{(n)})^m).$$
By induction, for $1 \leq k \leq m$, 
$$P_{n,m}^{(k)} = z^{m-k} Q_k (z^m),$$
where the polynomials $(Q_k)_{1 \leq k \leq m}$ have degree $n-1$ and satisfy $Q_1 = m Q_0'$, and 
$$Q_{k+1} (z) = (m-k) Q_k(z) + m z Q_k'(z). $$
for $1 \leq k \leq m-1$. 
 By Lemma 3.1. of \cite{galligo2025dynamics} and induction, all polynomials $(Q_k)_{1 \leq k \leq m}$ have real and positive roots, which provides the general form of the factorization of the polynomial $P_{n,m}^{(m)}$. 

Let us now assume $r_j^{(n)} \leq r_j$ for $1 \leq j \leq n$: in this case, the $j$-th smallest root 
of $Q_0$ is bounded by $r_j^m$ for $1 \leq j \leq n$.  
Iterating $m$ times Lemma \ref{impr_upper_bnd},
we deduce, for $1 \leq j \leq n-1$, 
successive upper bounds 
on the $j$-th smallest root 
of $Q_1, Q_2, \dots, Q_m$, from the 
fact that the $(j+1)$-th smallest root  
of $Q_0$ is at most $r_{j+1}^m$. 
More precisely, we get by induction, that
for $1 \leq k \leq m$, the $j$-th smallest root 
of $Q_k$ is at most 
$$ \left(1 - \frac{\alpha^m}{j +3 + 2 y_m} \right)
\left(1 - \frac{\alpha^m}{j + 2 + 2 y_m} \right)^{k-1} r^m_{j+1}. 
$$
Notice that in this induction, we use the fact that the ratio between these upper bounds for consecutive values of $j$ always remains bounded 
by $\alpha^m$, which is true because 
$r_j^m /r_{j+1}^m \leq \alpha^m$ by assumption, 
and the factors in front of $r_{j+1}^m$ are increasing in $j$.  
 
Since the $j$-th smallest root of $Q_m$ is
$(R_j^{(n)})^m$, we have 
$$ (R_j^{(n)})^m
\leq  \left(1 - \frac{\alpha^m}{j +3 + 2 y_m} \right)
\left(1 - \frac{\alpha^m}{j + 2 + 2 y_m} \right)^{m-1} r_{j+1}^m
\leq \left(1 - \frac{\alpha^m}{j +3 + 2 y_m} \right)^m r_{j+1}^m$$
 
which proves the claimed upper bound.  
\end{proof}

\begin{lemma}\label{lem1.9}
   Let $(m_n)_{n \geq 1}$ be a sequence of positive integers, such that 
   $m_n/\log n$ tends to infinity when $n \rightarrow \infty$.
   We keep the notation of Lemma \ref{lemma:improved-m-steps}. Then, for all integers $n \geq 2$, $1 \leq \ell \leq n-1$, 
$0 \leq q \leq m_n-1$, the roots of the $(\ell m_n - q)$-th derivative of $P_{n,m_n}$ are
\begin{itemize}
    \item $0$ with multiplicity $q$;
    \item $s_j e^{2\pi i k / m_n}$ for $1 \le j \le n-\ell$ and $0 \le k \le m_n-1$,
\end{itemize}
where the moduli $s_j$ satisfy, for $1 \le j \le n-\ell$,
\[
r_{j+\ell}^{(n)} e^{-\varepsilon_n} \frac{j-n \varepsilon_n}{j+\ell- n \varepsilon_n} \mathbf{1}_{j \geq n \varepsilon_n} \;\le\; s_j \;\le\;  r_{j+\ell}^{(n)} e^{\varepsilon_n}  \Bigl( \frac{j+ 1+  n \varepsilon_n}{j+\ell+ n \varepsilon_n } \Bigr)^{e^{- \varepsilon_n}},
\]
 where $(\varepsilon_n)_{n \geq 2}$ is a positive sequence depending only on the sequence $(m_n)_{n \geq 1}$ and tending to zero at infinity. 
\end{lemma}

\begin{proof}

Let us first assume $q = 0$. 
   Iterating Lemma \ref{lemma:improved-m-steps}, 
we deduce that for $1 \leq \ell \leq n-1$, 
$$P_{n,m_n}^{(\ell m_n)} (z)
= \frac{(nm_n)!}{((n-\ell)m_n)!} 
\prod_{j=1}^{n-\ell} (z^{m_n} -  (r_j^{(n, m_n, \ell)} )^{m_n} )$$
where for $1 \leq j \leq n- \ell$,
$$r_j^{(n, m_n,\ell)} 
\leq r_{j+\ell} \prod_{s = 1}^{\ell} 
\left(1 - \frac{\alpha^{m_n}}{ j+s+2 + 2y_{m_n}} \right) 
$$
as soon as $r_j^{(n)} \leq r_j$ for $1 \leq j \leq n$. 
Notice that iteration of Lemma \ref{lemma:improved-m-steps}
is possible because the ratio between the upper bounds obtained for indices $j$ and $j+1$ always remains smaller than or equal to $\alpha$, since the multiplicative factors 
$1 -  \alpha^{m_n}/(j+s+2 + 2y_{m_n}) $ are increasing with respect to $j$. 

We deduce 
$$r_j^{(n,m_n,\ell)} \leq r_{j+\ell} \exp \left( - \alpha^{m_n} \sum_{s = 1}^{\ell} 
 \frac{1}{ j+s+2 + 2y_{m_n}} \right) 
 \leq r_{j+\ell} \exp \left( - \alpha^{m_n} \int_{1}^{\ell+1} 
 \frac{dx}{ j+x+2 + 2y_{m_n}} \right)$$
 and 
 $$ r_j^{(n,m_n,\ell)} \leq r_{j+\ell} \left( \frac{ j + 3  + 2 y_{m_n}}{j +\ell + 3 + 2 y_{m_n}} \right)^{\alpha^{m_n}} $$
For $\gamma_n > 1$, we can apply this 
result to $r_j = r_j^{(n)} \gamma_n^j$, 
in which case we can take $\alpha = \gamma_n^{-1}$, and then 
$$y_{m_n} =  \frac{2 + \log_- (m_n \log \gamma_n)}{m_n \log \gamma_n} =: h(\gamma_n, m_n)  ,$$
which implies 
$$ r_j^{(n,m_n,\ell)} \leq r^{(n)}_{j+\ell} \gamma_n^{j+\ell} \left( \frac{ j + 3  + 2 h(\gamma_n, m_n)}{j +\ell + 3 + 2 h(\gamma_n, m_n)} \right)^{\gamma_n^{-m_n}}.$$
We now choose, for $n \geq 2$, 
$$\gamma_n = \exp \left( \frac{\phi(n) \log n}{n m_n} \right)$$
where $(\phi (n))_{n \geq 2}$ is a sequence of positive integers going to infinity sufficiently slowly, in such a way that $\gamma_n^n$ and $\gamma_n^{m_n}$ both tend to $1$ when $n \rightarrow \infty$. We can take, for example,
$$\phi(n) = 1 + \lfloor \sqrt{\min(n,m_n)/\log n} \rfloor,$$
which tends to infinity since $m_n/\log n \rightarrow \infty$ by assumption. In particular, we may assume that the sequence $(\phi(n))_{n \geq 2}$ depends only on the sequence $(m_n)_{n \geq 1}$.

We have $$m_n \log \gamma_n = \frac{\phi (n) \log n}{n} $$
and then 
$$h(\gamma_n, m_n) = \frac{2 + (\log n - \log \phi(n) - \log \log n)_+}{(\phi(n) \log n)/n} \leq \frac{2 - \log \log 2 + \log n}{(\phi(n) \log n)/n}
\leq \frac{5 n}{\phi(n)}$$
 since $2 - \log \log 2 \leq 4 \log 2 \leq 4 \log n$. 
We deduce
$$ r_j^{(n,m_n,\ell)} \leq r^{(n)}_{j+\ell} \gamma_n^{n} \left( \frac{ j + 3  + 10 n/\phi(n)}{j +\ell + 3 + 10 n/\phi(n)} \right)^{\gamma_n^{-m_n}}.$$
 Taking 
$$\varepsilon_n =  \max(\log (\gamma_n^n), \log (\gamma_n^{m_n}), 
3/n + 10/\phi(n)),$$
we get 
\begin{equation} r_j^{(n,m_n,\ell)} \leq r_{j+\ell}^{(n)} e^{\varepsilon_n}  \Bigl( \frac{j+ n \varepsilon_n}{j+\ell+ n \varepsilon_n } \Bigr)^{e^{- \varepsilon_n}} \label{xx1234}
\end{equation}
 which gives the upper bound of the lemma when $q = 0$. 
For the lower bound, we proceed as follows. If $r_j^{(n)} \geq r_j$
for $1 \leq j \leq n$, we get 
$$r_j^{(n,m_n,  \ell)} 
\geq r_{j+\ell} \prod_{s = 1}^{\ell} 
\left( 1 - \frac{1}{ \max \left(1,  j +s - 2 -  y_{m_n}   \right)} \right)$$
for $1 \leq j \leq n - \ell$. 
For $\gamma_n > 1$, we apply this result to $r_j = r_j^{(n)} \gamma_n^{j-n}$, which again allows to take $\alpha =  \gamma_n^{-1}$ and then $y_{m_n} = h(\gamma_n, m_n)$, which gives 
$$r_j^{(n,m_n,  \ell)} 
\geq r_{j+\ell} \prod_{s = 1}^{\ell} 
\left( 1 - \frac{1}{ \max \left(1,  j +s - 2 -  h(\gamma_n, m_n)   \right)} \right).$$
If $j \leq 2 + h(\gamma_n, m_n)$, the factor $s =1$ of the product vanishes, and then the lower bound is trivial. 
If $j \geq 2 + h (\gamma_n, m_n)$, we can discard the maximum with $1$, and then we get a telescopic product. Combining the two cases gives the bound
$$r_j^{(n,m_n,  \ell)} 
\geq r_{j+\ell} \, \frac{j- 2 - h(\gamma_n, m_n)}{j + \ell - 2 - h (\gamma_n, m_n)} \mathbf{1}_{j \geq 2 + h (\gamma_n, m_n)}.$$
With the previous choice of $\gamma_n$, we get, for $n \geq 2$, 
$r_{j+ \ell} \geq r_{j+\ell}^{(n)} \gamma_n^{-n} $
and $h(\gamma_n, m_n) \leq 5n /\phi(n)$, which implies 
$$r_j^{(n,m_n,  \ell)} 
\geq r^{(n)}_{j+\ell} \gamma_n^{-n} \, \frac{j- 2 -5 n /\phi(n)}{j + \ell - 2 - 5n/\phi(n)} \mathbf{1}_{j \geq 2 + 5 n /\phi(n)}.$$
Since the previous choice of $\varepsilon_n$ satisfies
$$\varepsilon_n \geq \max( \log (\gamma_n^n), 2/n + 5 /\phi(n)),$$
we get the lower bound claimed in the lemma, for $q = 0$.  

In the general case $0 \leq q \leq m_n -1$, we can write, for $1 \leq \ell \leq n-1$,
$$P_{n,m_n}^{(\ell m_n - q)} (z) 
= \frac{(nm_n)!}{((n-\ell)m_n + q)!}
z^q \prod_{j=1}^{n-\ell}
(z^{m_n} - (r_j^{(n,m_n,\ell-q/m_n)})^{m_n}),
$$
where, from the bound given by Lemma
\ref{impr_upper_bnd}, 
$$r_j^{(n,m_n, \ell)} \leq 
r_j^{(n,m_n,\ell-q/m_n)} \leq r_{j+1}^{(n,m_n,\ell-1)},$$
where $r_{j+1}^{(n,m_n,0)} = r_{j+1}^{(n)}$ in the case $\ell = 1$.
The lower bound of the lemma is then immediately deduced from the case $q = 0$. If $\ell \geq 2$, the upper bound is deduced from \eqref{xx1234} after replacing $j$ by $j+1$ and $\ell$ by $\ell -1$. If $\ell =1$, the upper bound of the lemma is trivial.

\end{proof}

\section{Statement and Proof of the main theorem}\label{sec:mainthm}
\begin{theorem}\label{thm:41improJ}
Let $ \nu_0 $ be a probability measure on $ \mathbb{R}_+ $ with compact support contained in 
$[0,A]$ for some $ A>0 $.  
For each $ n\ge 1 $ let $ (r_j^{(n)})_{1\le j\le n} $ be an increasing sequence in 
$ \mathbb{R}_+ $ such that
\begin{align*}
\frac{1}{n}\sum_{j=1}^n\delta_{r_j^{(n)}}\;\xrightarrow[n\to\infty]{\text{weakly}}\;\nu_0 .
\end{align*}
Let $ (m_n)_{n\ge1} $ be a sequence of positive integers satisfying
\begin{align*}
\frac{m_n}{\log n}\;\xrightarrow[n\to\infty]{}\; \infty .
\end{align*}

Define the polynomial of degree $N_n=n m_n$
\begin{align*}
P_{n,m_n}(z)=\prod_{j=1}^n\bigl(z^{m_n}-(r_j^{(n)})^{m_n}\bigr)
           =\prod_{j=1}^n\prod_{k=0}^{m_n-1}\!\Bigl(z-r_j^{(n)}e^{2\pi i k/m_n}\Bigr).
\end{align*}

For a fixed $t\in(0,1)$ consider the $\lfloor n m_n t\rfloor$-th derivative  
$ P_{n,m_n}^{(\lfloor n m_n t\rfloor)} $.

Then the empirical measure of its roots converges weakly to
\begin{align*}
\mu_t=\nu_t\otimes\unif,
\end{align*}
where
\begin{itemize}
\item $\unif$ is the uniform probability measure on the unit circle,
\item $\nu_t$ is the probability distribution of the random variable
  \begin{align*}
  \Bigl(1-\frac{t}{V_t}\Bigr)q_{\nu_0}(V_t),
  \end{align*}
  with $ V_t\sim\mathrm{Uniform}[t,1] $ and $ q_{\nu_0} $ the quantile function 
  of $\nu_0$, given by 
  $$q_{\nu_0} (\alpha) = \inf \{ y \geq 0, \nu_0 ([0,y]) \geq \alpha \}.$$
\end{itemize}

Moreover, the quantile function of the sub-probability measure $(1-t)\nu_t$ 
(which has total mass $1-t$) satisfies
\begin{equation*}
q_{(1-t)\nu_t}(x)=\frac{x}{x+t}\,q_{\nu_0}(x+t),\qquad 0\le x\le 1-t.
\end{equation*}

\end{theorem}

\begin{remark}
The measure $\nu_t$ is the same as the limiting measure found in Theorem 4.1. of \cite{galligo2025dynamics}: it depends only on $\nu_0$ and $t$. 
Therefore, under suitable assumptions detailed in Theorem 4.1. of \cite{galligo2025dynamics}, the same partial differential equations are satisfied for the distribution function $\Psi_t$ and the density $\psi_t$ of $(1-t)\nu_t$, namely
\begin{equation*} 
\frac{\partial\Psi_t(x)}{\partial t}
   =x\,\frac{\frac{\partial\Psi_t(x)}{\partial x}}{\Psi_t(x)}-1,
\end{equation*}
and 
\begin{equation*} 
\frac{\partial\psi}{\partial t}(x,t)
   =\frac{\partial}{\partial x}\!\Bigl(\frac{\psi(x,t)}
          {\,\frac{1}{x}\int_0^x\psi(y,t)\,dy\,}\Bigr).
\end{equation*}
for $t\in[0,1),\;x\in(0,A(1-t))$.

\end{remark}

\begin{proof}[Proof of Theorem \ref{thm:41improJ}]
    Let $t \in (0,1)$ be fixed, we write
\begin{align*}
\lfloor n m_n t \rfloor = \ell m_n - q,
\end{align*}
where $ 0 \leq q \leq m_n - 1 $ and $ \ell \in \mathbb{Z}_{\geq 0} $. Since
\begin{align*}
\ell m_n - q \leq n m_n t < \ell m_n - q + 1,
\end{align*}
we have
\begin{align*}
n t + \frac{q-1}{m_n} < \ell \leq n t + \frac{q}{m_n}.
\end{align*}
Thus, $ \ell = n t + \mathcal{O}(1) $. In particular, for $n \geq 2$ larger than a threshold depending only on $t$, $ 1 \leq \ell \leq n-1 $.

By Lemma \ref{lem1.9}, the empirical measure 
of the roots of $ P_{n,m_n}^{(\lfloor n m_n t\rfloor)} $ can be 
written as 
$$\frac{1}{q + m_n (n- \ell)} 
\left( q \delta_0 +  \sum_{j=1}^{n-\ell} \sum_{k=0}^{m_n-1}
\delta_{s_j e^{2 i \pi k/m_n}} \right)
$$
where 
\[
r_{j+\ell}^{(n)} e^{-\varepsilon_n} \frac{j-n \varepsilon_n}{j+\ell- n \varepsilon_n} \mathbf{1}_{j \geq n \varepsilon_n} \;\le\; s_j \;\le\;  r_{j+\ell}^{(n)} e^{\varepsilon_n}  \Bigl( \frac{j+1 +  n \varepsilon_n}{j+\ell+ n \varepsilon_n } \Bigr)^{e^{- \varepsilon_n}},
\]

In the case where $j \geq 1 + \sqrt{n} + n \sqrt{\varepsilon_n}$, 
we have necessarily $\varepsilon_n < 1$ since $j \leq n$, and
$n \varepsilon_n \leq (j - 1) \sqrt{\varepsilon_n}$, which implies 
\begin{align*} \frac{j-n \varepsilon_n}{j + \ell - n \varepsilon_n} 
 \mathbf{1}_{j \geq n \varepsilon_n} 
& \geq \frac{(j- 1) ( 1- \sqrt{\varepsilon_n})}{j - 1 + nt + \mathcal{O}(1)}
=  \frac{(j- 1) ( 1- \sqrt{\varepsilon_n})}{(j - 1 + nt) (1 + \mathcal{O}(1/(j-1)))}
\\ & =  \frac{(j- 1) ( 1- \sqrt{\varepsilon_n})}{(j - 1 + nt) (1 + \mathcal{O}(1/\sqrt{n}))},
\end{align*}
and then 
$$s_j \geq c_n r_{j+\ell}^{(n)} \frac{j-1}{j-1 + nt}, $$
for a sequence $(c_n)_{n \geq 2}$ in $[0,1)$, tending to $1$ at infinity, and depending only on the sequence $(\varepsilon_n)_{n \geq 2}$. 
Notice that $(\varepsilon_n)_{n \geq 2}$ depends only on $(m_n)_{n \geq 1}$ by Lemma \ref{lem1.9}.  

Similarly, 
\begin{align*}
\frac{j+1 +  n \varepsilon_n}{j+\ell+ n \varepsilon_n }
& \leq \frac{(j-1)( 1+ 2/(j-1) + \sqrt{\varepsilon_n})}{j-1 + nt + \mathcal{O}(1) }
\leq  \frac{(j- 1) ( 1 + 2/\sqrt{n} +  \sqrt{\varepsilon_n})}{(j - 1 + nt) (1 + \mathcal{O}(1/\sqrt{n}))},
\end{align*}
and since $j \leq n- \ell$ and $\varepsilon_n < 1$, 
$$ \left(\frac{j+\ell +  n \varepsilon_n}{j+1+ n \varepsilon_n } \right)^{1 - e^{- \varepsilon_n}}
\leq  \left( \frac{ 2n}{n \varepsilon_n} \right)^{1 - e^{- \varepsilon_n}} \leq (2/\varepsilon_n)^{\varepsilon_n}.
$$
Multiplying these two estimates, we deduce that 
$$s_j \leq C_n r^{(n)}_{j+\ell} \frac{j-1}{j-1 + nt}.$$
where $(C_n)_{n \geq 2}$ is a sequence in $(1, \infty)$ tending to $1$ and depending only on 
$(\varepsilon_n)_{n \geq 2}$. 

If, in the empirical measure of the roots of 
$ P_{n,m_n}^{(\lfloor n m_n t\rfloor)} $, we replace 
$s_j$ by 
$r_{j+\ell}^{(n)} (j-1)/(j-1 + nt)$, all roots corresponding to indices $j \geq 1 + \sqrt{n} + n \sqrt{\varepsilon_n}$ 
such that $r_{j+\ell}^{(n)} \leq A+1$ are moved by at most 
$\max(C_n-1, 1- c_n) (A + 1) $, for $n \geq 2$ larger than a threshold depending only on $t$. 
Since by assumption, 
\begin{align*}
\frac{1}{n}\sum_{j=1}^n\delta_{r_j^{(n)}}\;\xrightarrow[n\to\infty]{\text{weakly}}\;\nu_0 .
\end{align*}
 and $\nu_0$ is supported in $[0,A]$, there 
 are $o(n)$ indices $j$ such that $r_{j+\ell}^{(n)} > A + 1$
 when $n \rightarrow \infty$. Hence, 
 since $ 1+\sqrt{n} + n \sqrt{\varepsilon_n} = o(n)$, the proportion of roots which are moved by 
 more than $\max(C_n-1, 1- c_n) (A + 1) $ when $s_j$ is changed to $r_{j+\ell}^{(n)} (j-1)/(j-1 + nt)$ tends to zero when $n \rightarrow \infty$. 

Since $c_n \rightarrow 0$ when $n \rightarrow \infty$, the L\'evy-Prokhorov distance between the empirical measure of 
the roots of $P_{n,m_n}^{(\lfloor n m_n t \rfloor)}$ and the points obtained from these roots by changing $s_j$ to $(j-1)/(j-1 + nt) r_{j+ \ell}^{(n)}$ tends to zero when $n \rightarrow \infty$. In order to show 
convergence of the empirical measure of the
roots of $P_{n,m_n}^{(\lfloor n m_n t \rfloor)}$, it is then enough 
to show convergence, when $n \rightarrow \infty$, 
of the measure 
$$\frac{1}{q + (n- \ell) m_{n}} \left( q \delta_0 + \sum_{j=1}^{n-\ell} \sum_{k=0}^{m_{n}-1}
\delta_{e^{2 i\pi k/m_n} r_{j+\ell}^{(n)} (j-1)/(j-1 + nt) } \right)
$$
towards $\mu_t$: 
notice that $q + (n-\ell) m_n = n m_n - \lfloor n m_n t\rfloor$ is the number of 
roots of $P_{n, m_n}^{(\lfloor n m_n t \rfloor)}$. 

We can rotate the measure by an 
angle between $0$ and $2 \pi /m_n$, keeping a L\'evy-Prokhorov distance tending to zero, since we move a proportion tending to one of the points 
by $\mathcal{O} ((A+1) /m_n)$. Averaging 
among the possible angles, it is enough to show the 
convergence 
$$\frac{1}{q + (n- \ell) m_{n}}
\left(q \delta_0 + m_n \sum_{j=1}^{n-\ell}
\delta_{ r_{j+\ell}^{(n)} (j-1)/(j-1 + nt) }
\right) \otimes unif
\underset{n \rightarrow \infty}{\longrightarrow } \mu_t = \nu_t \otimes unif$$
i.e. 
$$\frac{1}{q + (n- \ell) m_{n}}
\left(q \delta_0 + m_n\sum_{j=1}^{n-\ell}
\delta_{ r_{j+\ell}^{(n)} (j-1)/(j-1 + nt) }
\right) 
\underset{n \rightarrow \infty}{\longrightarrow }  \nu_t.$$
By moving a negligible part of the measure, one deduces that it is enough to prove 
$$\frac{1}{n - \ell} 
\sum_{j=1}^{n-\ell} \delta_{ r_{j+\ell}^{(n)} (j-1)/(j-1 + nt) }\underset{n \rightarrow \infty}{\longrightarrow }  \nu_t.$$
The left-hand side is the distribution of 
$$r^{(n)}_{\ell + 1 + \lfloor (n-\ell) U \rfloor}
\frac{\lfloor (n-\ell) U \rfloor}{\lfloor (n-\ell) U \rfloor + nt}
= q_{\nu^{(n)}} \left( \frac{ \ell + 1 + \lfloor (n-\ell) U \rfloor}{n} \right)
\frac{\lfloor (n-\ell) U \rfloor}{\lfloor (n-\ell) U \rfloor + nt}
$$
where $U$ is uniformly distributed on $[0,1]$, 
and $\nu^{(n)}$ is the empirical distribution 
of $(r^{(n)}_j)_{1 \leq j \leq n}$, i.e. 
$$\nu^{(n)}  = \frac{1}{n} \sum_{j=1}^{n} \delta_{r^{(n)}_j}. $$
Since $nt - 1 \leq \ell \leq nt + 1$, 
we get for $n > 3/t$, and then $3/n < t$, 
$$ q_{\nu^{(n)}} \left(   t + (1-t)U - 3/n   \right)
 \leq q_{\nu^{(n)}} \left( \frac{ \ell + 1 + \lfloor (n-\ell) U \rfloor}{n} \right)
\leq q_{\nu^{(n)}} \left( \min(1, t + (1-t)U + 3/n)   \right). $$
Now, using L\'evy-Prokhorov distance 
and convergence of $(\nu^{(n)})_{n \geq 1}$
towards $\nu_0$, we deduce that 
for fixed $\alpha \in [0,1]$, $\varepsilon > 0$, 
and for $n$ large enough, 
$$ \nu^{(n)} ([0, q_{\nu_0} (\alpha) + \varepsilon]) + \varepsilon \geq \nu_0 ([0, q_{\nu_0} (\alpha)]) \geq \alpha, \;
\nu_0 ([0, q_{\nu^{(n)}} (\alpha) + \varepsilon]) + \varepsilon \geq \nu^{(n)} ([0, q_{\nu^{(n)}} (\alpha)]) \geq \alpha,
$$
and then 
$$ q_{\nu^{(n)}} (\max(0,\alpha - \varepsilon))
\leq q_{\nu_0} (\alpha) + \varepsilon, \; 
q_{\nu^{(n)}}(\alpha) \geq   q_{\nu_0} (\max(0,\alpha - \varepsilon)) - \varepsilon.
$$
Hence, for all $\varepsilon \in (0,t)$, 
$$\underset{n \rightarrow \infty}{\lim \inf} \, 
q_{\nu^{(n)}} \left( t + (1-t)U - 3/n   \right)
\geq \underset{n \rightarrow \infty}{\lim \inf} \, 
q_{\nu^{(n)}} \left( t + (1-t)U - \varepsilon/2  \right)
\geq   
q_{\nu_0}  \left( t + (1-t)U - \varepsilon \right)
- \varepsilon/2,
$$
and then, letting $\varepsilon \rightarrow 0$, 
$$\underset{n \rightarrow \infty}{\lim \inf} \, 
q_{\nu^{(n)}} \left( t + (1-t)U - 3/n   \right)
\geq q_{\nu_0}  \left( ( t + (1-t)U)-  \right)
$$
where $q_{\nu_0}(\alpha-)$ is the limit of $q_{\nu_0} (\beta)$ when $\beta$ tends to $\alpha$ from below. 
Similarly, if $t + (1-t)U \leq 1 - \varepsilon$, 
we get 
\begin{align*} &  \underset{n \rightarrow \infty}{\lim \sup} \, q_{\nu^{(n)}} \left( \min(1, t + (1-t)U + 3/n)   \right) \\ & \leq \underset{n \rightarrow \infty}{\lim \sup} \, q_{\nu^{(n)}} \left(   t + (1-t)U + \varepsilon/2   \right) \leq q_{\nu_0}  \left( t + (1-t)U + \varepsilon \right)
+ \varepsilon/2,
\end{align*}
and then, letting $\varepsilon \rightarrow 0$, 
we deduce that for $U <1$, and then almost surely, 
$$ \underset{n \rightarrow \infty}{\lim \sup} \, q_{\nu^{(n)}} \left( \min(1, t + (1-t)U + 3/n)   \right) \leq q_{\nu_0} \left( ( t + (1-t)U)+  \right),$$
where $q_{\nu_0}(\alpha+)$ is the limit of $q_{\nu_0} (\beta)$ when $\beta$ tends to $\alpha$ from above. 
Now, $q_{\nu_0}$ is nondecreasing and then has 
at most countably many discontinuities, which implies that almost surely, 
$$  q_{\nu_0} \left( ( t + (1-t)U)+  \right)
=  q_{\nu_0} \left( ( t + (1-t)U)-  \right)
=  q_{\nu_0} \left(  t + (1-t)U  \right).$$
Hence, almost surely,
\begin{align*} 
 q_{\nu_0} & \left(  t + (1-t)U  \right)
  =  q_{\nu_0} \left( ( t + (1-t)U)-  \right)
 \leq \underset{n \rightarrow \infty}{\lim \inf} \, 
q_{\nu^{(n)}} \left( t + (1-t)U - 3/n   \right)
\\ & \leq \underset{n \rightarrow \infty}{\lim \inf} \, q_{\nu^{(n)}} \left( \frac{ \ell + 1 + \lfloor (n-\ell) U \rfloor}{n} \right)
  \leq \underset{n \rightarrow \infty}{\lim \sup} \, q_{\nu^{(n)}} \left( \frac{ \ell + 1 + \lfloor (n-\ell) U \rfloor}{n} \right)
\\ & \leq \underset{n \rightarrow \infty}{\lim \sup} \,
q_{\nu^{(n)}} \left( \min(1, t + (1-t)U + 3/n)   \right)   \leq q_{\nu_0} \left( ( t + (1-t)U)+  \right) 
=  q_{\nu_0} \left(  t + (1-t)U  \right).
\end{align*}
Since $\ell = nt + \mathcal{O}(1)$, we also have 
$$\frac{\lfloor (n-\ell) U \rfloor}{\lfloor (n-\ell) U \rfloor + nt} \underset{n \rightarrow 
\infty}{\longrightarrow}  \frac{ (1-t)U}{t + (1-t)U},$$ 
and then 
$$ q_{\nu^{(n)}} \left( \frac{ \ell + 1 + \lfloor (n-\ell) U \rfloor}{n} \right) \frac{\lfloor (n-\ell) U \rfloor}{\lfloor (n-\ell) U \rfloor + nt}
\underset{n \rightarrow 
\infty}{\longrightarrow} 
 q_{\nu_0} \left(  t + (1-t)U  \right) \frac{ (1-t)U}{t + (1-t)U}
$$
almost surely. Letting 
$V_t := t + (1-t) U$, which is uniformly distributed 
on $[t,1]$, 
this proves convergence of 
the empirical measure of the roots of the $\lfloor n m_n t \rfloor$-th derivative of  
$P_{n,m_n}$ towards $\nu_t \otimes unif$.

\end{proof}

\section{Concluding Remarks}\label{sec:conclude}
Our main result establishes that the convergence of the empirical root measure to the predicted limit~$\mu_t$ holds under the condition $m_n/\log n \to \infty$, significantly improving the previous requirement $m_n/(n\log n) \to \infty$ from~\cite{galligo2025dynamics}. This fully resolves the conjecture stated therein regarding the robustness of the deterministic structured sampling scheme.

Several natural questions arise regarding further extensions.

\medskip
\noindent\textbf{The regime $m_n \to \infty$ with no rate.}
Our proof requires $m_n/\log n \to \infty$ primarily to ensure that the regularization parameter satisfies $\gamma_n^n \to 1$. It is natural to expect that the same limiting dynamics holds as soon as $m_n \to \infty$, regardless of the rate. Establishing this would likely require a fundamentally different approach to controlling error propagation, perhaps through probabilistic or potential-theoretic arguments that bypass multiplicative regularization entirely.

\medskip
\noindent\textbf{The regime $m_n = m$ fixed.}
When $m_n = m$ is a fixed constant independent of~$n$, the situation is qualitatively different, and the same limiting dynamics should not be expected to hold in general. For $m = 1$ or $m = 2$, the polynomial $P_{n,m}(z) = \prod_{j=1}^n (z - r_j^{(n)})$ has exclusively real roots, placing the problem in the setting of real-rooted polynomials where the Steinerberger PDE on the real line governs the dynamics---a fundamentally different equation from the radial PDE studied here. 

More generally, for fixed $m \geq 3$, the polynomial $P_{n,m}$ places $m$ roots on each of $n$ concentric circles. After $\lfloor nmt \rfloor$ differentiations, one expects a limiting measure that depends not only on~$\nu_0$ and~$t$, but also on~$m$ itself. The key observation is that for fixed~$m$, the discrete rotational symmetry of order~$m$ is never well approximated by full rotational invariance: the interaction between roots on the same circle, separated by an angular distance of~$2\pi/m$, cannot be neglected at leading order. It would be interesting to characterize the limiting dynamics in this regime. One expects a family of evolution equations or quantile relations parametrized by~$m$, interpolating between the real-line dynamics (small~$m$) and the rotationally invariant dynamics ($m \to \infty$). Even the explicit identification of the limiting measure for specific values of~$m$ and simple choices of~$\nu_0$ appears to be a nontrivial open problem.

\medskip
\noindent\textbf{Other directions.}
The technique developed in Section~\ref{sec:sing-Step}, which provides sharper control of rational sums evaluated near their zeros, may find applications beyond the rotationally invariant setting---for instance, in the analysis of deterministic polynomial configurations with other discrete symmetries, or in the study of spectral dynamics for structured random matrices under analogous operations.

\bibliographystyle{abbrv}
\bibliography{RP3}
\end{document}